\documentclass[12pt,reqno]{amsart}
\allowdisplaybreaks[4]
\usepackage{microtype}
\usepackage{tikz}
\usepackage{xcolor}
\usepackage{latexsym,amsmath,amssymb,epic}
\usepackage{latexsym,amssymb,epic}
\usepackage{amsmath,amsthm}
\usepackage{color}
\usepackage[bottom=3.5cm,left=3.2cm,right=3.2cm]{geometry}
\usepackage{etoolbox}
\usepackage[shortlabels]{enumitem}
\usepackage[noadjust]{cite}
\usepackage{hyperref}
\hypersetup{colorlinks=true, 
	linktoc=all, 
	linkcolor=blue} 

\numberwithin{equation}{section}

\theoremstyle{plain}
\newtheorem{theorem}{Theorem}[section]
\newtheorem*{theorem*}{Theorem}

\newtheorem{lemma}[theorem]{Lemma}

\theoremstyle{definition}

\newtheorem*{example*}{Example}
\newtheorem{conjecture}[theorem]{Conjecture}
\newtheorem*{conjecture*}{Conjecture}
\newtheorem{remark}[theorem]{Remark}
\newtheorem*{remark*}{Remark}
\newtheorem*{remarks*}{Remarks}

\makeatletter
\patchcmd{\@settitle}{\uppercasenonmath}{\boldmath\uppercasenonmath}{}{}
\patchcmd{\section}{\scshape}{\bfseries\boldmath}{}{}
\patchcmd{\subsection}{\bfseries}{\bfseries\boldmath}{}{}
\renewcommand{\@secnumfont}{\bfseries}
\makeatother

\makeatletter
\renewcommand{\maketag@@@}[1]{\hbox{\normalsize\normalfont#1}}%
\@namedef{subjclassname@2020}{\textup{2020} Mathematics Subject Classification}

\makeatother

\begin{document}

\title[Lambert series and double Lambert series]
{Identities and transformations for Lambert series and double Lambert series}

\author[S.-P. Cui]{Su-Ping Cui}
\address[Su-Ping Cui]{School of Mathematical Sciences, Qufu Normal University,
Qufu, Shandong 810008, P.R. China}
\email{jiayoucui@163.com}

\author[D. Tang]{Dazhao Tang}

\address[Dazhao Tang]{School of Mathematical Sciences, Chongqing Normal University,
Chongqing 401331, P.R. China}
\email{dazhaotang@sina.com}


\begin{abstract}
We establish two identities for Lambert series and double Lambert series, thereby
resolving conjectures of Andrews, Dixit, Schultz and Yee (Acta Arith.~181:253--286,
2017), as well as Amdeberhan, Andrews and Ballantine (J Combin Theory Series A
221:106154, 2026). The proofs are based on classical transformations in the theory
of infinite series together with a systematic rearrangement of double Lambert series.
\end{abstract}



\maketitle

\section{Introduction}

In 1771, J. H. Lambert \cite{Lam1965} demonstrated that the generating function of the
divisor function can be expressed as
\begin{align*}
\sum_{n=1}^\infty\dfrac{q^n}{1-q^n}.
\end{align*}
This expression naturally leads to the notion of a Lambert series, which in general
takes the form
\begin{align*}
\sum_{n=1}^\infty\dfrac{a_n\!\:q^n}{1-q^n},
\end{align*}
where the sequence of coefficients ${a_n}$ governs the arithmetic and analytic
properties of the series. Of particular interest is the case $a_n=n^{2k-1}$, for which
the resulting Lambert series is a modular form, or a quasimodular form when $k=1$. A
comprehensive account of Lambert series identities and their applications may be found
in the survey of Schmidt \cite{Sch2020}. Building on this classical framework, recent
developments have significantly expanded the scope of Lambert series. Notably,
Amdeberhan, Andrews and Ballantine \cite{AAB2026} introduced a substantial
generalization by considering generalized Lambert series of the form
\begin{align*}
\sum_{n=1}^\infty R_{n}(q^n,q),
\end{align*}
where $R_n(x,y)$ is a rational function in two variables. Furthermore, they proposed a
two-parameter extension, termed a double Lambert series \cite{AAB2026}, given by
\begin{align*}
\sum_{n=1}^\infty S_{n,m}(q^n,q^m,q),
\end{align*}
where $S_{n,m}(x,y,z)$ is a rational function in three variables. These generalizations
not only broaden the classical theory but also reveal new structural connections with
partition theory and modular phenomena.

Parallel to these analytic developments, deep connections between $q$-series and
partition theory have continued to emerge. In 2015, Andrews, Dixit and Yee
\cite{ADY2015} discovered a novel partition-theoretic interpretation for the
coefficients of the Ramanujan/Watson mock theta function
$\omega(q)$, defined by
\begin{align*}
\omega(q):=\sum_{n=0}^\infty\dfrac{q^{2n^2+2n}}{(q;q^2)_{n+1}^2},
\end{align*}
where $q$ is a complex number satisfying $|q|<1$, and the standard $q$-Pochhammer
symbols are given by
\begin{align*}
(a;q)_n &:=\prod_{j=0}^{n-1}(1-aq^j),\\
(a;q)_\infty &:=\prod_{j=0}^{\infty}(1-aq^j).
\end{align*}
If $p_\omega(n)$ denotes the coefficient of $q^n$ in the series
$q\omega(q)$, then $p_\omega(n)$ enumerates the number of partitions of $n$ in
which all odd parts are less than twice the smallest part. This interpretation was
subsequently extended by Andrews, Dixit, Schultz and Yee \cite{ADSY2017} to the
setting of overpartitions. Recall that an overpartition \cite{CL2004} is a partition
in which the first (equivalently, the last) occurrence of each distinct part may be
overlined.

Motivated by this extension, Andrews et al.\ \cite{ADSY2017} introduced
$\overline{p}_\omega(n)$ to denote the number of
overpartitions of $n$ such that all odd parts are less than twice the smallest part,
with the additional condition that the smallest part is always overlined. They
established several striking congruence properties for $\overline{p}_\omega(n)$. In
particular, using $q$-series techniques, they \cite[Theorem~1.4]{ADSY2017} proved that
for any $n\geq0$,
\begin{align}
\overline{p}_\omega(4n+3) &\equiv0\pmod{4},\label{ADSY-cong-1}\\
\overline{p}_\omega(8n+6) &\equiv0\pmod{4}.\label{ADSY-cong-2}
\end{align}
At the conclusion of their paper, Andrews et al.\ \cite{ADSY2017} remarked that an
alternative proof of \eqref{ADSY-cong-1} and \eqref{ADSY-cong-2} could be derived
assuming the validity of the following conjecture.

\begin{conjecture}{\cite[Problem~2]{ADSY2017}}\label{ADSY-conj}
Let
\begin{align*}
Y(q):=\sum_{n=1}^\infty\sum_{m=1}^\infty\dfrac{(-1)^mq^{2nm+m}}{(1-q^{2m-1})(1+q^n)}.
\end{align*}
Then $Y(q)$ is an odd function of $q$.
\end{conjecture}

The first objective of the present paper is to prove the Andrews--Dixit--Schultz--Yee
(ADSY) conjecture. More precisely, we derive a new representation of $Y(q)$, from
which its parity property follows immediately.

\begin{theorem}\label{Main-THM}
We have
\begin{align}
Y(q)=-q\!\;\dfrac{(q^4;q^4)_\infty^4}{(q^2;q^2)_\infty^2}\sum_{k=1}^{\infty}
\dfrac{q^{2k}}{1+q^{2k}}.\label{exp-Y-q}
\end{align}
\end{theorem}

Amdeberhan, Andrews and Ballantine \cite{AAB2026} systematically investigated the
interplay between classical Lambert series, multiple Lambert series and classical
$q$-series of the Rogers--Ramanujan type. In the course of their study, they observed
that attempts to prove Conjecture \ref{ADSY-conj} naturally led to two further
conjectures. The second objective of this paper is to confirm one of these
conjectures, which we state below.

\begin{theorem}{\emph{\cite[Conjecture~5.12]{AAB2026}}}\label{Main-THM-2}
If $r$ is a positive integer, then
\begin{align}
[q^{2r}]\!\:\sum_{n=1}^\infty\sum_{m=1}^\infty\dfrac{q^{2mn}}{(1+q^{2n-1})
(1-q^{2m-1})}=[q^{2r}]\!\:\sum_{n=1}^\infty\dfrac{(n-1)q^n}{1+q^{2n-1}}.
\label{AAB-conj}
\end{align}
\end{theorem}

The remainder of this paper is organized as follows. In Sect.~\ref{sect:proof-1}, we
first establish two key lemmas and then present the proof of Theorem \ref{Main-THM}.
Building on two additional important lemmas and Theorem \ref{Main-THM}, we prove
Theorem \ref{Main-THM-2} in Sect.~\ref{sect:proof-2}.

\section{Proof of Theorem \ref{Main-THM}}\label{sect:proof-1}

In this section, we present the proof of Theorem \ref{Main-THM}.

We begin by introducing three auxiliary double Lambert series, defined by
\begin{align*}
X(q) &=\sum_{k=1}^\infty\sum_{n=1}^\infty\dfrac{(-1)^kq^{2kn+k}}
{(1-q^n)(1-q^{2k-1})},\\
Z(q) &=\sum_{k=1}^\infty\sum_{n=1}^\infty\dfrac{q^{2kn}}{(1+q^{2n-1})(1-q^{2k-1})},\\
A(q) &=\sum_{k=1}^\infty\sum_{n=k}^\infty\dfrac{q^{k+n}}{(1+q^n)(1+q^{2k-1})}.
\end{align*}
These series will play a central role in decomposing and reorganizing the double sum
defining $Y(q)$.

We first establish the following auxiliary identity that expresses $Y(q)$ in terms of
the series $Z(q)$.

\begin{lemma}
\begin{align}
Y(q)=-\sum_{k=1}^\infty\sum_{n=1}^\infty\dfrac{q^{k+n}}{(1+q^n)(1+q^{2k-1})}
-\sum_{k=1}^\infty\sum_{n=1}^\infty\dfrac{q^{2kn+k}}{(1-q^{2n})(1+q^{2k-1})}+Z(q).
\label{Aux-iden-1}
\end{align}
\end{lemma}

\begin{proof}
Amdeberhan et al.\ \cite[Proposition~5.3]{AAB2026} proved that
\begin{align}
Y(q)=-\sum_{k=2}^\infty\dfrac{q^k}{1+q^{2k-1}}\sum_{n=1}^{k-1}\dfrac{q^n}{1+q^n}.
\label{AAB-iden}
\end{align}
Rewriting \eqref{AAB-iden} in a more convenient form, we obtain that
\begin{align}
Y(q) &=-\sum_{k=1}^\infty\dfrac{q^k}{1+q^{2k-1}}\sum_{n=1}^{k-1}\dfrac{q^n}{1+q^n}
\nonumber\\
 &=-\sum_{k=1}^\infty\dfrac{q^k}{1+q^{2k-1}}\Bigg(\sum_{n=1}^\infty\dfrac{q^n}
{1+q^n}-\sum_{n=k}^\infty\dfrac{q^n}{1+q^n}\Bigg)\nonumber\\
 &=-\sum_{k=1}^\infty\sum_{n=1}^\infty\dfrac{q^{k+n}}{(1+q^n)(1+q^{2k-1})}+A(q),
\label{step-1}
\end{align}
where the term $A(q)$ is isolated for further analysis.

We now examine $A(q)$ more closely. A sequence of elementary manipulations yields
\begin{align}
A(q) &=\sum_{k=1}^\infty\dfrac{q^k}{1+q^{2k-1}}\sum_{m=0}^\infty\dfrac{q^{m+k}}
{1+q^{m+k}}\nonumber\\
 &=\sum_{k=1}^\infty\dfrac{q^{2k}}{1+q^{2k-1}}\sum_{m=0}^\infty q^m\sum_{n=0}^\infty
(-1)^nq^{mn+kn}\nonumber\\
 &=\sum_{k=1}^\infty\dfrac{q^{2k}}{1+q^{2k-1}}\sum_{n=0}^\infty(-1)^nq^{kn}
\sum_{m=0}^\infty q^{(n+1)m}\nonumber\\
 &=\sum_{k=1}^\infty\dfrac{q^{2k}}{1+q^{2k-1}}\sum_{n=0}^\infty\dfrac{(-1)^nq^{kn}}
{1-q^{n+1}}.\label{step-2}
\end{align}
Next, shifting the index of summation in the inner sum by replacing $n$ with $n-1$ in
the final expression of \eqref{step-2}, we obtain that
\begin{align}
A(q) &=\sum_{k=1}^\infty\sum_{n=1}^\infty\dfrac{(-1)^{n-1}q^{kn+k}}
{(1-q^n)(1+q^{2k-1})}\nonumber\\
 &=-\sum_{k=1}^\infty\sum_{n=1}^\infty\dfrac{q^{2kn+k}}{(1-q^{2n})(1+q^{2k-1})}
+\sum_{k=1}^\infty\sum_{n=1}^\infty\dfrac{q^{2kn}}{(1-q^{2n-1})(1+q^{2k-1})}\nonumber\\
 &=-\sum_{k=1}^\infty\sum_{n=1}^\infty\dfrac{q^{2kn+k}}{(1-q^{2n})(1+q^{2k-1})}+Z(q).
\label{step-3}
\end{align}
Combining \eqref{step-1} and \eqref{step-3} immediately yields the desired identity
\eqref{Aux-iden-1}.
\end{proof}

We next establish the following lemma, which plays a crucial role in the proof of
Theorem \ref{Main-THM}.

\begin{lemma}
\begin{align}
Z(q)=\sum_{k=1}^\infty\sum_{n=1}^\infty\dfrac{q^{k+n}}{(1+q^{2n-1})(1-q^{2k})}
+\dfrac{1}{2}X(q)+\dfrac{1}{2}Y(q).\label{Aux-iden-2}
\end{align}
\end{lemma}

\begin{proof}
We begin by rewriting $Z(q)$ in several equivalent forms. First, we observe that
\begin{align}
Z(q) &=\sum_{n=1}^\infty\dfrac{1}{1+q^{2n-1}}\sum_{k=1}^\infty\dfrac{q^{2kn}}
{1-q^{2k-1}}\nonumber\\
 &=\sum_{n=1}^\infty\dfrac{1}{1+q^{2n-1}}\sum_{k=1}^\infty q^{2kn}
\sum_{m=0}^\infty q^{2km-m}\nonumber\\
 &=\sum_{n=1}^\infty\dfrac{1}{1+q^{2n-1}}\sum_{m=0}^\infty q^{-m}\sum_{k=1}^\infty
q^{(2m+2n)k}\nonumber\\
 &=\sum_{n=1}^\infty\dfrac{1}{1+q^{2n-1}}\sum_{m=0}^\infty q^{-m}\,
\dfrac{q^{2m+2n}}{1-q^{2m+2n}}\nonumber\\
 &=\sum_{n=1}^\infty\dfrac{1}{1+q^{2n-1}}\sum_{m=0}^\infty\dfrac{q^{m+2n}}
{1-q^{2m+2n}}.\label{step-4}
\end{align}
In the final expression of \eqref{step-4}, we introduce the change of variables
$k=m+n$, which allows us to rewrite the double sum in terms of the new index $k$.
This yields
\begin{align}
Z(q) &=\sum_{n=1}^\infty\dfrac{1}{1+q^{2n-1}}\sum_{k=n}^\infty\dfrac{q^{k+n}}
{1-q^{2k}}\nonumber\\
 &=\sum_{n=1}^\infty\dfrac{q^n}{1+q^{2n-1}}\sum_{k=n}^\infty\dfrac{q^k}{1-q^{2k}}
\label{Z-new-exp}\\
 &=\sum_{k=1}^\infty\dfrac{q^k}{1-q^{2k}}\sum_{n=1}^k\dfrac{q^n}{1+q^{2n-1}}\nonumber\\
 &=\sum_{k=1}^\infty\dfrac{q^k}{1-q^{2k}}\Bigg(\sum_{n=1}^\infty\dfrac{q^n}
{1+q^{2n-1}}-\sum_{n=k+1}^\infty\dfrac{q^n}{1+q^{2n-1}}\Bigg)\nonumber\\
 &=\sum_{k=1}^\infty\sum_{n=1}^\infty\dfrac{q^{k+n}}{(1+q^{2n-1})(1-q^{2k})}
-\sum_{k=1}^\infty\dfrac{q^k}{1-q^{2k}}\sum_{n=k+1}^\infty\dfrac{q^n}{1+q^{2n-1}}
\nonumber\\
 &=\sum_{k=1}^\infty\sum_{n=1}^\infty\dfrac{q^{k+n}}{(1+q^{2n-1})(1-q^{2k})}
-\dfrac{1}{2}\sum_{k=1}^\infty\dfrac{q^k}{1-q^k}\sum_{n=k+1}^\infty
\dfrac{q^n}{1+q^{2n-1}}\nonumber\\
 &\quad-\dfrac{1}{2}\sum_{k=1}^\infty\dfrac{q^k}{1+q^k}\sum_{n=k+1}^\infty
\dfrac{q^n}{1+q^{2n-1}}\nonumber\\
 &=\sum_{k=1}^\infty\sum_{n=1}^\infty\dfrac{q^{k+n}}{(1+q^{2n-1})(1-q^{2k})}
-\dfrac{1}{2}\sum_{k=1}^\infty\dfrac{q^k}{1-q^k}\sum_{n=k+1}^\infty
\dfrac{q^n}{1+q^{2n-1}}+\dfrac{1}{2}Y(q),\label{step-5}
\end{align}
where the last step in \eqref{step-5} follows from interchanging the order of
summation in \eqref{AAB-iden}.

We now turn to the series $X(q)$. A sequence of straightforward transformations shows
that
\begin{align}
X(q) &=\sum_{k=1}^\infty\sum_{n=1}^\infty\dfrac{(-1)^kq^{2kn+k}}
{(1-q^n)(1-q^{2k-1})}\nonumber\\
 &=\sum_{n=1}^\infty\dfrac{1}{1-q^n}\sum_{k=1}^\infty(-1)^kq^{2kn+k}
\sum_{j=0}^\infty q^{(2k-1)j}\nonumber\\
 &=\sum_{n=1}^\infty\dfrac{1}{1-q^n}\sum_{j=0}^\infty q^{-j}\sum_{k=1}^\infty
(-1)^kq^{(2n+2j+1)k}\nonumber\\
 &=\sum_{n=1}^\infty\dfrac{1}{1-q^n}\sum_{j=0}^\infty q^{-j}\,\dfrac{-q^{2n+2j+1}}
{1+q^{2n+2j+1}}\nonumber\\
 &=-\sum_{n=1}^\infty\dfrac{1}{1-q^n}\sum_{j=0}^\infty\dfrac{q^{2n+j+1}}
{1+q^{2n+2j+1}}\nonumber\\
 &=-\sum_{n=1}^\infty\dfrac{q^n}{1-q^n}\sum_{j=1}^\infty\dfrac{q^{n+j}}{1+q^{2n+2j-1}}
\nonumber\\
 &=-\sum_{n=1}^\infty\dfrac{q^n}{1-q^n}\sum_{k=n+1}^\infty\dfrac{q^k}{1+q^{2k-1}}.
\label{step-6}
\end{align}

Finally, combining \eqref{step-5} and \eqref{step-6} yields \eqref{Aux-iden-2},
completing the proof.
\end{proof}

We are now in a position to complete the proof of Theorem \ref{Main-THM}.

\begin{proof}[Proof of Theorem \ref{Main-THM}]
We begin by replacing $q$ with $-q$ in the definition of $Z(q)$. A direct inspection
shows that
\begin{align}
Z(-q)=\sum_{k=1}^\infty\sum_{n=1}^\infty\dfrac{q^{2kn}}{(1-q^{2n-1})(1+q^{2k-1})}
=Z(q).\label{step-7}
\end{align}

Next, substituting $-q$ for $q$ in \eqref{Aux-iden-1} and invoking \eqref{step-7}, we
obtain that
\begin{align}
Y(-q) &=-\sum_{k=1}^\infty\sum_{n=1}^\infty\dfrac{(-1)^{k+n}q^{k+n}}{(1+(-q)^n)
(1-q^{2k-1})}-\sum_{k=1}^\infty\sum_{n=1}^\infty\dfrac{(-1)^kq^{2kn+k}}
{(1-q^{2n})(1-q^{2k-1})}\nonumber\\
 &\quad+Z(-q)\nonumber\\
 &=-\sum_{k=1}^\infty\sum_{n=1}^\infty\dfrac{(-1)^{k+n}q^{k+n}}{(1+(-q)^n)(1-q^{2k-1})}
-\dfrac{1}{2}\sum_{k=1}^\infty\sum_{n=1}^\infty\dfrac{(-1)^kq^{2kn+k}}
{(1-q^n)(1-q^{2k-1})}\nonumber\\
 &\quad-\dfrac{1}{2}\sum_{k=1}^\infty\sum_{n=1}^\infty\dfrac{(-1)^kq^{2kn+k}}
{(1+q^n)(1-q^{2k-1})}+Z(q)\nonumber\\
 &=-\sum_{k=1}^\infty\sum_{n=1}^\infty\dfrac{(-1)^{k+n}q^{k+n}}{(1+(-q)^n)(1-q^{2k-1})}
-\dfrac{1}{2}X(q)-\dfrac{1}{2}Y(q)+Z(q).\label{step-8}
\end{align}
Combining \eqref{Aux-iden-2} with \eqref{step-8} then yields
\begin{align}
Y(-q) &=\sum_{n=1}^\infty\dfrac{q^n}{1+q^{2n-1}}\sum_{k=1}^\infty\dfrac{q^k}{1-q^{2k}}
-\sum_{n=1}^\infty\dfrac{(-q)^n}{1-q^{2n-1}}\sum_{k=1}^\infty\dfrac{(-1)^kq^k}
{1+(-q)^k}.\label{step-9}
\end{align}

We now appeal to the identity established in \cite[p.~20]{AAB2026},
\begin{align}
\sum_{n=1}^\infty\dfrac{q^n}{1+q^{2n-1}}=q\!\;\dfrac{(q^4;q^4)_\infty^4}
{(q^2;q^2)_\infty^2},\label{step-10}
\end{align}
which is readily seen to be equivalent to
\begin{align}
\sum_{n=1}^\infty\dfrac{(-q)^n}{1-q^{2n-1}}=-q\!\;\dfrac{(q^4;q^4)_\infty^4}
{(q^2;q^2)_\infty^2}.\label{step-11}
\end{align}
Substituting \eqref{step-10} and \eqref{step-11} into \eqref{step-9}, we deduce that
\begin{align}
Y(-q) &=q\!\;\dfrac{(q^4;q^4)_\infty^4}{(q^2;q^2)_\infty^2}\Bigg(\sum_{k=1}^\infty
\dfrac{q^k}{1-q^{2k}}+\sum_{k=1}^\infty\dfrac{(-1)^kq^k}{1+(-q)^k}\Bigg)\nonumber\\
 &=q\!\;\dfrac{(q^4;q^4)_\infty^4}{(q^2;q^2)_\infty^2}\Bigg(\sum_{k=1}^\infty
\dfrac{q^k}{1-q^{2k}}-\sum_{k=1}^\infty\dfrac{q^{2k-1}}{1-q^{2k-1}}
+\sum_{k=1}^\infty\dfrac{q^{2k}}{1+q^{2k}}\Bigg)\nonumber\\
 &=q\!\;\dfrac{(q^4;q^4)_\infty^4}{(q^2;q^2)_\infty^2}\Bigg(\sum_{k=1}^\infty
\sum_{m=0}^\infty q^{(2m+1)k}-\sum_{k=1}^\infty\dfrac{q^{2k-1}}{1-q^{2k-1}}
+\sum_{k=1}^\infty\dfrac{q^{2k}}{1+q^{2k}}\Bigg)\nonumber\\
 &=q\!\;\dfrac{(q^4;q^4)_\infty^4}{(q^2;q^2)_\infty^2}\Bigg(\sum_{m=0}^\infty
\sum_{k=1}^\infty q^{(2m+1)k}-\sum_{k=1}^\infty\dfrac{q^{2k-1}}{1-q^{2k-1}}
+\sum_{k=1}^\infty\dfrac{q^{2k}}{1+q^{2k}}\Bigg)\nonumber\\
 &=q\!\;\dfrac{(q^4;q^4)_\infty^4}{(q^2;q^2)_\infty^2}\Bigg(\sum_{m=0}^\infty
\dfrac{q^{2m+1}}{1-q^{2m+1}}-\sum_{k=1}^\infty\dfrac{q^{2k-1}}{1-q^{2k-1}}
+\sum_{k=1}^\infty\dfrac{q^{2k}}{1+q^{2k}}\Bigg)\nonumber\\
 &=q\!\;\dfrac{(q^4;q^4)_\infty^4}{(q^2;q^2)_\infty^2}\sum_{k=1}^\infty
\dfrac{q^{2k}}{1+q^{2k}}.\label{step-12}
\end{align}
Finally, replacing $-q$ by $q$ in \eqref{step-12} yields the desired identity, thereby
completing the proof of Theorem \ref{Main-THM}.
\end{proof}

\section{Proof of Theorem \ref{Main-THM-2}}\label{sect:proof-2}

This section is devoted to the proof of Theorem \ref{Main-THM-2}.

To facilitate the argument, we introduce the following auxiliary functions, which
will be used frequently in the subsequent analysis:
\begin{align*}
N(q) &=\sum_{k=1}^\infty\sum_{n=1}^\infty\dfrac{(-1)^kq^{2kn+k+n}}{(1+q^{2n-1})
(1-q^{2k-1})},\\
D(q) &=\sum_{k=1}^\infty\sum_{n=1}^{k-1}\dfrac{q^k}{(1-q^{2n})(1+q^{2k-1})},\\
L(q) &=\sum_{k=2}^\infty\dfrac{(k-1)q^k}{1+q^{2k-1}}
=\sum_{k=1}^\infty\dfrac{(k-1)q^k}{1+q^{2k-1}}.
\end{align*}

We begin by establishing the following auxiliary identity, which relates $Y(q)$,
$Z(q)$, $L(q)$ and $D(q)$ to certain infinite series.

\begin{lemma}
\begin{align}
Y(q)+L(q)=D(q)+Z(q)-q\!\;\dfrac{(q^4;q^4)_\infty^4}{(q^2;q^2)_\infty^2}
\sum_{n=1}^\infty\dfrac{q^n}{1-q^{2n}}.\label{Aux-iden-3}
\end{align}
\end{lemma}

\begin{proof}
Starting from \eqref{AAB-iden}, we rewrite $Y(q)$ in a form suitable for further
manipulation:
\begin{align*}
Y(q) &=-\sum_{k=1}^\infty\dfrac{q^k}{1+q^{2k-1}}\sum_{n=1}^{k-1}\dfrac{q^n}{1+q^n}\\
 &=-\sum_{k=1}^\infty\dfrac{q^k}{1+q^{2k-1}}\sum_{n=1}^{k-1}\dfrac{\{(1+q^n)-1\}}
{1+q^n}\\
 &=-\sum_{k=1}^\infty\dfrac{q^k}{1+q^{2k-1}}\sum_{n=1}^{k-1}1
+\sum_{k=1}^\infty\dfrac{q^k}{1+q^{2k-1}}\sum_{n=1}^{k-1}\dfrac{1}{1+q^n}\\
 &=-\sum_{k=2}^\infty\dfrac{(k-1)q^k}{1+q^{2k-1}}+\sum_{k=1}^\infty
\dfrac{q^k}{1+q^{2k-1}}\sum_{n=1}^{k-1}\dfrac{1}{1+q^n}.
\end{align*}
By the definition of $L(q)$, this yields
\begin{align*}
Y(q)+L(q) &=\sum_{k=1}^\infty\dfrac{q^k}{1+q^{2k-1}}\sum_{n=1}^{k-1}
\dfrac{1}{1+q^n}\\
 &=\sum_{k=1}^\infty\dfrac{q^k}{1+q^{2k-1}}\sum_{n=1}^{k-1}\dfrac{1}{1-q^{2n}}
-\sum_{k=1}^\infty\dfrac{q^k}{1+q^{2k-1}}\sum_{n=1}^{k-1}\dfrac{q^n}{1-q^{2n}}\\
 &=\sum_{k=1}^\infty\dfrac{q^k}{1+q^{2k-1}}\sum_{n=1}^{k-1}\dfrac{1}{1-q^{2n}}
-\sum_{k=1}^\infty\dfrac{q^k}{1+q^{2k-1}}\sum_{n=1}^\infty\dfrac{q^n}{1-q^{2n}}\\
 &\quad+\sum_{k=1}^\infty\dfrac{q^k}{1+q^{2k-1}}\sum_{n=k}^\infty
\dfrac{q^n}{1-q^{2n}}.
\end{align*}
Invoking \eqref{step-10} together with the definition of $D(q)$, we deduce that
\begin{align}
Y(q)+L(q) &=D(q)-q\!\;\dfrac{(q^4;q^4)_\infty^4}{(q^2;q^2)_\infty^2}
\sum_{n=1}^\infty\dfrac{q^n}{1-q^{2n}}+\sum_{k=1}^\infty\dfrac{q^k}{1+q^{2k-1}}
\sum_{n=k}^\infty\dfrac{q^n}{1-q^{2n}}.\label{step-13}
\end{align}
On the other hand, by \eqref{Z-new-exp} we have
\begin{align}
Z(q)=\sum_{n=1}^\infty\dfrac{q^n}{1+q^{2n-1}}\sum_{k=n}^\infty\dfrac{q^k}{1-q^{2k}}.
\label{step-14}
\end{align}
Comparing \eqref{step-13} and \eqref{step-14}, we immediately obtain the desired
identity \eqref{Aux-iden-3}.
\end{proof}

The following lemma plays a central role in the proof of Theorem \ref{Main-THM-2}.

\begin{theorem}
\begin{align}
D(q)+D(-q) &=2\!\;\dfrac{(q^4;q^4)_\infty^4}{(q^2;q^2)_\infty^2}\sum_{n=1}^\infty
\dfrac{q^{2n}}{1-q^{4n-2}}.\label{Aux-iden-4}
\end{align}
\end{theorem}

\begin{proof}
We begin by rewriting $D(q)$ in a more tractable form. A sequence of manipulations
yields
\begin{align*}
D(q) &=\sum_{n=1}^\infty\dfrac{1}{1-q^{2n}}\sum_{k=n+1}^\infty\dfrac{q^k}
{1+q^{2k-1}}\\
 &=\sum_{n=1}^\infty\dfrac{1}{1-q^{2n}}\sum_{m=0}^\infty\dfrac{q^{n+m+1}}
{1+q^{2n+2m+1}}\\
 &=\sum_{n=1}^\infty\dfrac{q^{n+1}}{1-q^{2n}}\sum_{m=0}^\infty q^m\sum_{k=0}^\infty
(-1)^kq^{2nk+2mk+k}\\
 &=\sum_{n=1}^\infty\dfrac{q^{n+1}}{1-q^{2n}}\sum_{k=0}^\infty(-1)^kq^{2nk+k}
\sum_{m=0}^\infty q^{(2k+1)m}\\
 &=\sum_{n=1}^\infty\dfrac{q^n}{1-q^{2n}}\sum_{k=0}^\infty
\dfrac{(-1)^kq^{2nk+k+1}}{1-q^{2k+1}}\\
 &=\sum_{n=1}^\infty\dfrac{1}{1-q^{2n}}\sum_{k=1}^\infty
\dfrac{(-1)^{k-1}q^{2nk-n+k}}{1-q^{2k-1}}\\
 &=-\sum_{k=1}^\infty\sum_{n=1}^\infty\dfrac{(-1)^kq^{2nk-n+k}}
{(1-q^{2n})(1-q^{2k-1})}.
\end{align*}
Consequently, we deduce that
\begin{align}
D(q) &=-\sum_{k=1}^\infty\dfrac{(-1)^kq^k}{1-q^{2k-1}}\sum_{n=1}^\infty
\dfrac{q^{2nk-n}}{1-q^{2n}}\nonumber\\
 &=-\sum_{k=1}^\infty\dfrac{(-1)^kq^k}{1-q^{2k-1}}\sum_{n=1}^\infty q^{2nk-n}
\sum_{m=0}^\infty q^{2mn}\nonumber\\
 &=-\sum_{k=1}^\infty\dfrac{(-1)^kq^k}{1-q^{2k-1}}\sum_{m=0}^\infty\sum_{n=1}^\infty
q^{(2k+2m-1)n}\nonumber\\
 &=-\sum_{k=1}^\infty\dfrac{(-1)^kq^k}{1-q^{2k-1}}\sum_{m=0}^\infty
\dfrac{q^{2k+2m-1}}{1-q^{2k+2m-1}}\nonumber\\
 &=-\sum_{k=1}^\infty\dfrac{(-1)^kq^k}{1-q^{2k-1}}\sum_{n=k}^\infty
 \dfrac{q^{2n-1}}{1-q^{2n-1}}.\label{step-15}
\end{align}

Next, we relate $D(q)$ to $L(-q)$. Starting from \eqref{step-15}, we write
\begin{align*}
D(q) &=-\sum_{k=1}^\infty\dfrac{(-1)^kq^k}{1-q^{2k-1}}\sum_{n=k}^\infty
 \dfrac{q^{2n-1}}{1-q^{2n-1}}\nonumber\\
 &=-\sum_{k=1}^\infty\dfrac{(-1)^kq^k}{1-q^{2k-1}}\sum_{n=1}^\infty
\dfrac{q^{2n-1}}{1-q^{2n-1}}+\sum_{k=1}^\infty\dfrac{(-1)^kq^k}{1-q^{2k-1}}
\sum_{n=1}^{k-1}\dfrac{q^{2n-1}}{1-q^{2n-1}}.
\end{align*}
Evaluating the first term and decomposing the second, we obtain that
\begin{align}
D(q) &=q\!\;\dfrac{(q^4;q^4)_\infty^4}{(q^2;q^2)_\infty^2}\sum_{n=1}^\infty
\dfrac{q^{2n-1}}{1-q^{2n-1}}+\sum_{k=1}^\infty\dfrac{(-1)^kq^k}{1-q^{2k-1}}
\sum_{n=1}^{k-1}\dfrac{\{(q^{2n-1}-1)+1\}}{1-q^{2n-1}}\nonumber\\
 &=\dfrac{(q^4;q^4)_\infty^4}{(q^2;q^2)_\infty^2}\sum_{n=1}^\infty\dfrac{q^{2n}}
{1-q^{2n-1}}-\sum_{k=1}^\infty\dfrac{(-1)^kq^k}{1-q^{2k-1}}
\sum_{n=1}^{k-1}1\nonumber\\
 &\hspace{11.4em}+\sum_{k=1}^\infty\dfrac{(-1)^kq^k}{1-q^{2k-1}}\sum_{n=1}^{k-1}
\dfrac{1}{1-q^{2n-1}}\nonumber\\
 &=\dfrac{(q^4;q^4)_\infty^4}{(q^2;q^2)_\infty^2}\sum_{n=1}^\infty\dfrac{q^{2n}}
{1-q^{2n-1}}-L(-q)+\sum_{n=1}^\infty\dfrac{1}{1-q^{2n-1}}\sum_{k=n+1}^\infty
\dfrac{(-1)^kq^k}{1-q^{2k-1}}\nonumber\\
 &=\dfrac{(q^4;q^4)_\infty^4}{(q^2;q^2)_\infty^2}\sum_{n=1}^\infty\dfrac{q^{2n}}
{1-q^{2n-1}}-L(-q)\nonumber\\
 &\hspace{11.4em}+\sum_{n=1}^\infty\dfrac{1}{1-q^{2n-1}}\sum_{m=0}^\infty
\dfrac{(-1)^{n+m+1}q^{n+m+1}}{1-q^{2n+2m+1}}.\label{step-16}
\end{align}

We now simplify the last term in \eqref{step-16}. A straightforward computation
shows that
\begin{align}
\sum_{n=1}^\infty &\,\dfrac{1}{1-q^{2n-1}}\sum_{m=0}^\infty
\dfrac{(-1)^{n+m+1}q^{n+m+1}}{1-q^{2n+2m+1}}\nonumber\\
 &=\sum_{n=1}^\infty\dfrac{(-1)^{n+1}q^{n+1}}{1-q^{2n-1}}\sum_{m=0}^\infty(-1)^mq^m
\sum_{k=0}^\infty q^{2nk+2mk+k}\nonumber\\
 &=\sum_{n=1}^\infty\dfrac{(-1)^{n+1}q^{n+1}}{1-q^{2n-1}}\sum_{k=0}^\infty q^{2nk+k}
\sum_{m=0}^\infty(-1)^mq^{(2k+1)m}\nonumber\\
 &=\sum_{n=1}^\infty\dfrac{(-1)^{n+1}q^{n+1}}{1-q^{2n-1}}\sum_{k=0}^\infty
\dfrac{q^{2nk+k}}{1+q^{2k+1}}\nonumber\\
 &=\sum_{n=1}^\infty\dfrac{(-1)^{n+1}}{1-q^{2n-1}}\sum_{k=1}^\infty
\dfrac{q^{2nk+k-n}}{1+q^{2k-1}}.\label{step-17}
\end{align}
Substituting \eqref{step-17} into \eqref{step-16}, we find that
\begin{align*}
D(q)=\dfrac{(q^4;q^4)_\infty^4}{(q^2;q^2)_\infty^2}\sum_{n=1}^\infty\dfrac{q^{2n}}
{1-q^{2n-1}} &-L(-q)\\
 &-\sum_{n=1}^\infty\sum_{k=1}^\infty\dfrac{(-1)^nq^{2nk+k-n}}
{(1+q^{2k-1})(1-q^{2n-1})},
\end{align*}
which is equivalent to
\begin{align}
D(-q)=\dfrac{(q^4;q^4)_\infty^4}{(q^2;q^2)_\infty^2}\sum_{n=1}^\infty\dfrac{q^{2n}}
{1+q^{2n-1}} &-L(q)\nonumber\\
 &-\sum_{n=1}^\infty\sum_{k=1}^\infty\dfrac{(-1)^kq^{2nk+k-n}}
{(1-q^{2k-1})(1+q^{2n-1})}.\label{step-18}
\end{align}

We restrict our attention to the last term in \eqref{step-18}. A series of
manipulation yields
\begin{align*}
\sum_{n=1}^\infty\sum_{k=1}^\infty &\,\dfrac{(-1)^kq^{2nk+k-n}}
{(1-q^{2k-1})(1+q^{2n-1})}\\
 &=\sum_{k=1}^\infty\dfrac{(-1)^kq^{2nk+k}}{1-q^{2k-1}}\sum_{n=1}^\infty
\dfrac{\{(q^{-n}+q^{n-1})-q^{n-1}\}}{1+q^{2n-1}}\\
 &=\sum_{k=1}^\infty\sum_{n=1}^\infty\dfrac{(-1)^kq^{2nk+k-n}}{1-q^{2k-1}}
-q^{-1}N(q)\\
 &=\sum_{k=1}^\infty\dfrac{(-1)^kq^k}{1-q^{2k-1}}\sum_{n=1}^\infty q^{(2k-1)n}
-q^{-1}N(q)\\
 &=\sum_{k=1}^\infty\dfrac{(-1)^kq^{3k-1}}{(1-q^{2k-1})^2}-q^{-1}N(q).
\end{align*}
On the other hand,
\begin{align*}
L(q) &=\sum_{n=0}^\infty\dfrac{n\!\;q^{n+1}}{1+q^{2n+1}}\\
 &=\sum_{n=0}^\infty n\!\;q^{n+1}\sum_{m=0}^\infty(-1)^mq^{(2n+1)m}\\
 &=\sum_{m=0}^\infty(-1)^mq^{m+1}\sum_{n=0}^\infty n\!\;q^{(2m+1)n}\\
 &=\sum_{m=0}^\infty(-1)^mq^{m+1}\!\;\dfrac{q^{2m+1}}{(1-q^{2m+1})^2}\\
 &=\sum_{m=1}^\infty\dfrac{(-1)^{m-1}q^{3m-1}}{(1-q^{2m-1})^2}.
\end{align*}
Therefore,
\begin{align}
\sum_{n=1}^\infty\sum_{k=1}^\infty\dfrac{(-1)^kq^{2nk+k-n}}{(1-q^{2k-1})
(1+q^{2n-1})}=-L(q)-q^{-1}N(q).\label{step-19}
\end{align}
Plugging \eqref{step-19} into \eqref{step-18}, we obtain that
\begin{align}
D(-q)=\dfrac{(q^4;q^4)_\infty^4}{(q^2;q^2)_\infty^2}\sum_{n=1}^\infty\dfrac{q^{2n}}
{1+q^{2n-1}}+q^{-1}N(q).\label{step-20}
\end{align}

It remains to evaluate $N(q)$. Starting from its definition, a routine calculation
yields
\begin{align}
N(q) &=\sum_{k=1}^\infty\dfrac{(-1)^kq^k}{1-q^{2k-1}}\sum_{n=1}^\infty
\dfrac{q^{(2k+1)n}}{1+q^{2n-1}}\nonumber\\
 &=\sum_{k=1}^\infty\dfrac{(-1)^kq^k}{1-q^{2k-1}}\sum_{n=1}^\infty q^{(2k+1)n}
\sum_{m=0}^\infty(-1)^mq^{(2n-1)m}\nonumber\\
 &=\sum_{k=1}^\infty\dfrac{(-1)^kq^k}{1-q^{2k-1}}\sum_{m=0}^\infty(-1)^mq^{-m}
\sum_{n=1}^\infty q^{(2k+2m+1)n}\nonumber\\
 &=\sum_{k=1}^\infty\dfrac{(-1)^kq^k}{1-q^{2k-1}}\sum_{m=0}^\infty(-1)^m
\dfrac{q^{2k+m+1}}{1-q^{2k+2m+1}}.\label{step-21}
\end{align}
Taking $k+m=n$ in the inner summation in the last equality of \eqref{step-21}, we
deduce that
\begin{align}
N(q) &=\sum_{k=1}^\infty\dfrac{(-1)^kq^{2k}}{1-q^{2k-1}}\sum_{n=k}^\infty
(-1)^{n-k}\dfrac{q^{n+1}}{1-q^{2n+1}}\nonumber\\
 &=\sum_{k=1}^\infty\dfrac{q^{2k}}{1-q^{2k-1}}\sum_{n=k+1}^\infty
\dfrac{(-1)^{n+1}q^n}{1-q^{2n-1}}\nonumber\\
 &=-\sum_{n=1}^\infty\dfrac{(-1)^nq^n}{1-q^{2n-1}}\sum_{k=1}^{n-1}
\dfrac{q^{2k}}{1-q^{2k-1}}\nonumber\\
 &=-\sum_{n=1}^\infty\dfrac{(-1)^nq^n}{1-q^{2n-1}}\sum_{k=1}^\infty
\dfrac{q^{2k}}{1-q^{2k-1}}+\sum_{n=1}^\infty\dfrac{(-1)^nq^n}{1-q^{2n-1}}
\sum_{k=n}^\infty\dfrac{q^{2k}}{1-q^{2k-1}}\nonumber\\
 &=q\!\;\dfrac{(q^4;q^4)_\infty^4}{(q^2;q^2)_\infty^2}\sum_{k=1}^\infty
\dfrac{q^{2k}}{1-q^{2k-1}}-qD(q),\label{step-22}
\end{align}
where the last step in \eqref{step-22} follows from \eqref{step-15}.

Finally, substituting \eqref{step-22} into \eqref{step-20}, we obtain that
\begin{align}
\sum_{n=1}^\infty\sum_{k=1}^\infty\dfrac{(-1)^kq^{2nk+k-n}}{(1-q^{2k-1})
(1+q^{2n-1})}
 &=D(q)-L(q)-\dfrac{(q^4;q^4)_\infty^4}{(q^2;q^2)_\infty^2}\sum_{k=1}^\infty
\dfrac{q^{2k}}{1-q^{2k-1}}.\label{step-23}
\end{align}
Substituting \eqref{step-23} into \eqref{step-18}, we obtain that
\begin{align*}
D(-q)=\dfrac{(q^4;q^4)_\infty^4}{(q^2;q^2)_\infty^2}\sum_{n=1}^\infty
\dfrac{q^{2n}}{1+q^{2n-1}}+\dfrac{(q^4;q^4)_\infty^4}{(q^2;q^2)_\infty^2}
\sum_{k=1}^\infty\dfrac{q^{2k}}{1-q^{2k-1}}-D(q),
\end{align*}
from which we conclude that
\begin{align*}
D(q)+D(-q) &=\dfrac{(q^4;q^4)_\infty^4}{(q^2;q^2)_\infty^2}\sum_{n=1}^\infty
\bigg(\dfrac{q^{2n}}{1+q^{2n-1}}+\dfrac{q^{2n}}{1-q^{2n-1}}\bigg)\\
 &=2\!\;\dfrac{(q^4;q^4)_\infty^4}{(q^2;q^2)_\infty^2}\sum_{n=1}^\infty
\dfrac{q^{2n}}{1-q^{4n-2}},
\end{align*}
which completes the proof.
\end{proof}

We are now in a position to prove Theorem \ref{Main-THM-2}.

\begin{proof}[Proof of Theorem \ref{Main-THM-2}]
Replacing $q$ by $-q$ in \eqref{Aux-iden-3} yields
\begin{align}
Y(-q)+L(-q)=D(-q)+Z(-q)+q\!\;\dfrac{(q^4;q^4)_\infty^4}{(q^2;q^2)_\infty^2}
\sum_{n=1}^\infty\dfrac{(-1)^nq^n}{1-q^{2n}}.\label{Aux-iden-3-dual}
\end{align}
It follows immediately from \eqref{exp-Y-q} that
\begin{align}
Y(-q)=-Y(q).\label{Use-iden-1}
\end{align}
By the definition of $Z(q)$, we have
\begin{align}
Z(q)=Z(-q).\label{Use-iden-2}
\end{align}
Combining \eqref{Aux-iden-3} and \eqref{Aux-iden-3-dual}, and using
\eqref{Use-iden-1} and \eqref{Use-iden-2}, we obtain that
\begin{align}
L(q)+L(-q) &=D(q)+D(-q)+2Z(q)\nonumber\\
 &\hspace{3.4em}-q\!\;\dfrac{(q^4;q^4)_\infty^4}{(q^2;q^2)_\infty^2}
\sum_{n=1}^\infty\bigg(\dfrac{q^n}{1-q^{2n}}-\dfrac{(-1)^nq^n}{1-q^{2n}}
\bigg)\nonumber\\
 &=D(q)+D(-q)+2Z(q)-2\!\;\dfrac{(q^4;q^4)_\infty^4}{(q^2;q^2)_\infty^2}
\sum_{n=1}^\infty\dfrac{q^{2n}}{1-q^{4n-2}}.\label{step-24}
\end{align}
Applying \eqref{Aux-iden-4} to \eqref{step-24}, we immediately find that
\begin{align*}
L(q)+L(-q)=2Z(q)=Z(q)+Z(-q),
\end{align*}
which establishes \eqref{AAB-conj}.
\end{proof}

\begin{remark}
The proof of \eqref{AAB-conj} relies heavily on Conjecture \ref{ADSY-conj}. Indeed, a
direct application of \eqref{Aux-iden-3} and \eqref{Aux-iden-3-dual} shows that
Conjecture \ref{ADSY-conj} follows immediately from \eqref{AAB-conj}.
Therefore, Conjecture \ref{ADSY-conj} and Theorem \ref{Main-THM-2} are, to some
extent, equivalent.
\end{remark}

\section{Concluding remarks}

In this paper, building on two identities due to Amdeberhan et al.\ \cite{AAB2026} and
combining them with systematic $q$-series manipulations, we derive an explicit
expression for a double Lambert series. Thisidentity yields a direct and transparent
proof of the ADSY conjecture concerning the parity of such series. As a further
application, we establish a conjecture of Amdeberhan, Andrews, and Ballantine relating
Lambert series and double Lambert series.

Beyond resolving these conjectures, our approach highlights the effectiveness of
exploiting structural relations among double Lambert series. In particular, it provides
a unified framework in which parity phenomena arise naturally from analytic identities,
rather than from ad hoc or case-dependent arguments. From this perspective, the present
work not only confirms the ADSY conjecture and the Amdeberhan--Andrews--Ballantine
conjecture, but also suggests that the methods developed here may be applicable to a
broader class of problems involving Lambert series and related $q$-series.

\section*{Acknowledgements}

Su-Ping Cui was partially supported by the National Natural Science Foundation of China
(No.~12001309).
Dazhao Tang was partially supported by the National Natural Science Foundation of China
(No.~12201093) and the Science and Technology Research Program of Chongqing Municipal
Education Commission (No.~KJQN202500501).

\bibliographystyle{amsplain}

\end{document}